\documentclass[11pt]{article}
\usepackage{amsmath,amsthm}
\usepackage{amssymb,mathrsfs}
\usepackage{bm,mathtools}
\usepackage{extarrows}
\usepackage{tikz}
\usetikzlibrary{matrix,calc}
\tikzset{
  chart/.style={
  rectangle, 
  rounded corners, 
  draw=black, semithick,
  text width=6.2em, 
  text centered}
}

\usepackage{xcolor}

\usepackage{geometry}
\usepackage[colorlinks=true,
linkcolor=blue,citecolor=blue,
urlcolor=blue,pagebackref]{hyperref}

\makeatletter
\def\@seccntDot{.}
\def\@seccntformat#1{\csname the#1\endcsname\@seccntDot\hskip 0.5em}
\renewcommand\section{\@startsection{section}{1}{\z@}%
{18\p@ \@plus 6\p@ \@minus 3\p@}%
{9\p@ \@plus 6\p@ \@minus 3\p@}%
{\large\bfseries\boldmath}}
\renewcommand\subsection{\@startsection{subsection}{2}{\z@}%
{12\p@ \@plus 6\p@ \@minus 3\p@}%
{3\p@ \@plus 6\p@ \@minus 3\p@}%
{\bfseries\boldmath}}
\renewcommand\subsubsection{\@startsection{subsubsection}{3}{\z@}%
{12\p@ \@plus 6\p@ \@minus 3\p@}%
{\p@}%
{\bfseries\boldmath}}
\makeatother

\usepackage{microtype}

\theoremstyle{plain}
\newtheorem{theorem}{Theorem}[section]
\newtheorem{lemma}{Lemma}[section]
\newtheorem{corollary}{Corollary}[section]
\newtheorem{proposition}{Proposition}[section]
\newtheorem{problem}{Problem}[section]

\theoremstyle{definition}

\newtheorem{remark}{Remark}[section]
\newtheorem{example}{Example}[section]

\numberwithin{equation}{section}
\allowdisplaybreaks
\DeclareMathOperator{\ex}{ex}

\DeclareMathOperator{\cl}{cl}
\DeclareMathOperator{\supp}{supp}

\title{A new spectral Tur\'an theorem for weighted graphs and consequences}

\author{Lele Liu\footnote{School of Mathematical Sciences, Anhui University, Hefei 230601,
P.R. China. E-mail: \texttt{liu@ahu.edu.cn} (L. Liu). Supported by the National
Nature Science Foundation of China (No. 12471320), and Anhui Provincial Natural Science Foundation for Excellent Young Scholars (No. 2408085Y003).}
~~ and ~~
Bo Ning\footnote{College of Cryptology and Cyber Science \& College of Computer Science, Nankai University, Tianjin 300350, P.R. China.
E-mail: \texttt{bo.ning@nankai.edu.cn} (B. Ning). Partially supported by the National Nature Science
Foundation of China (No. 12371350) and Fundamental Research Funds for the Central Universities, Nankai University (No. 63243151).}}

\date{}

\begin{document}
\maketitle

\begin{abstract}
Confirming a conjecture of Elphick and Edwards and strengthening a spectral theorem 
of Wilf, Nikiforov proved that for any $K_{r+1}$-free graph $G$, 
$\lambda(G)^2 \leq 2 (1 - 1/r) m$, where $\lambda(G)$ is the spectral radius 
of $G$, and $m$ is the number of edges of $G$. This result was later improved 
in \cite{LiuN26}, which showed that for any graph $G$,
$\lambda(G)^2 \leq 2 \sum_{e \in E(G)} \frac{\cl(e) - 1}{\cl(e)}$,
where $\cl(e)$ denotes the order of the largest clique containing the edge $e$. 

In this paper, we further extend this inequality to weighted graphs, proving that
\[
\lambda(G)^2 \leq 2 \sum_{e \in E(G)} \frac{\cl(e) - 1}{\cl(e)} w(e)^2,
\]
and we characterize all extremal graphs attaining this bound. Our main theorem yields 
several new consequences, including two vertex-based and vertex-degree-based local 
versions of Tur\'an's theorem, as well as weighted generalizations of the Edwards--Elphick 
theorem and the Cvetkovi\'c theorem, and two localized versions of Wilf's theorem. 
One of these localized Wilf's theorem confirms a conjecture that originates from 
Probability and Operator Algebras and was proposed by R. Tripathi independently of us. 
Moreover, our main result unifies and implies numerous earlier ones from spectral 
graph theory and extremal graph theory, including Stanley's spectral inequality, 
Hong's inequality, a localized Tur\'an-type theorem, and a recent extremal theorem by Adak 
and Chandran. Notably, while Nikiforov's earlier spectral inequality implied Stanley's bound, 
it did not imply Hong's inequality---a gap that is now bridged by our result. 
As a key tool, we establish the inequality $\sum_{e \in E(G)} \frac{2}{\cl(e)} \geq n-1$, 
which complements an upper bound $\sum_{e \in E(G)} \frac{2}{\cl(e)-1} \leq n^2 - 2m$ 
due to Brada\v{c}, and Malec and Tompkins, independently.
\par\vspace{2mm}

\end{abstract} 


\section{Introduction}
\label{sec:1}

A \emph{weighted graph} $G$ is a triple $(V(G), E(G), w)$, where $V(G)$ is the vertex set, 
$E(G)$ is the edge set, and $w: E(G)\to \mathbb{R}$ is a weight function assigning a real value to each edge.
When all edge weights are $1$, this reduces to the usual unweighted graph. The \emph{adjacency matrix} of
a weighted graph $G$ on $n$ vertices is the $n\times n$ matrix $A(G) = [a_{ij}]$, where $a_{ij} = w(e)$
if $e=ij\in E(G)$, and $0$ otherwise. The \emph{spectral radius} of $G$, denoted by $\lambda(G)$, is the 
maximum of the absolute values of the eigenvalues of $A(G)$. For a matrix $M\in\mathbb{R}^{m\times n}$, 
the \emph{Frobenius norm} of $M = [m_{ij}]$ is defined as $\|M\|_F := \big(\sum_{i=1}^m \sum_{j=1}^n m_{ij}^2\big)^{1/2}$.
Let $A = [a_{ij}]$ and $B = [b_{ij}]$ be two matrices of the same dimensions. 
The \emph{Hadamard product} $A\circ B$ of $A$ and $B$ is defined via element-wise multiplication: 
$(A\circ B)_{ij} = a_{ij}b_{ij}$. For a subset $S\subseteq [n]:=\{1,2,\ldots,n\}$, 
the \emph{characteristic vector} of $S$ is the vector $\bm{1}_S\in\{0,1\}^n$ whose $i$-th 
entry is $1$ if $i\in S$ and $0$ otherwise. 

For a weighted graph $G$ and an edge $e \in E(G)$, let $\cl_G(e)$ (or $\cl(e)$ in short) 
denote the order of the largest clique containing $e$. Our main theorem is the following.

\begin{theorem}\label{thm:main}
Let $G$ be a weighted graph with edge weight $w(e)$ for each edge $e$. Then
\begin{equation}\label{eq:weighted-local-bound}
\lambda(G) \leq \sqrt{2 \sum_{e\in E(G)} \frac{\cl(e) - 1}{\cl(e)} w(e)^2}.
\end{equation}
Equality holds if and only if $G$ is, up to isolated vertices, a complete $r$-partite graph 
for some $r$ \textup{(}with partition $V_1\cup V_2\cup\cdots\cup V_r$\textup{)}, and 
there exists a vector $\bm{w}\in\mathbb{R}^n$ such that 
\begin{enumerate}
\item[$(1)$] $A(G) = \pm\sum_{i=1}^r (\bm{1}_{V_i}\circ\bm{w}) \big((\bm{1} - \bm{1}_{V_i})\circ\bm{w}\big)^{\mathrm{T}}$; and

\item[$(2)$] $\|\bm{1}_{V_i}\circ\bm{w}\|^2 = \|\bm{w}\|^2 - \sqrt{1-1/r} \cdot \|A(G)\|_F$ for each $i\in [r]$.
\end{enumerate}
\end{theorem}

We now present an example to clarify the conditions for equality in \eqref{eq:weighted-local-bound}.

\begin{example}\label{example}
Consider the weighted complete $3$-partite graph $G$ with vertex classes $V_1 = \{v_1\}$, $V_2 = \{v_2,v_3\}$, 
and $V_3 = \{v_4,v_5\}$. The edge weights are indicated by the numbers on the edges (see Figure \ref{fig:example-graph}). 
\begin{figure}[htbp]
\centering
\begin{tikzpicture}[scale=0.8,thick]
\coordinate (v1) at (0,0);
\coordinate (v2) at (-2.4,1.6);
\coordinate (v3) at (2.4,-1.6);
\coordinate (v4) at (2.4,1.6);
\coordinate (v5) at (-2.4,-1.6);
\foreach \i in {1,2,3,4,5}
{
\filldraw (v\i) circle (0.1);
} 
\draw (v1) -- (v2) -- (v4) -- (v3) -- (v5) -- (v2);
\draw (v3) -- (v1) -- (v5);
\draw (v1) -- (v4);
\node[above=2.5pt] at (v1) {$v_1$};
\node[left=2.5pt] at (v2) {$v_2$};
\node[left=2.5pt] at (v5) {$v_5$};
\node[right=2.5pt] at (v4) {$v_4$};
\node[right=2.5pt] at (v3) {$v_3$};
\node[below=2.5pt] at ($(v1)!0.5!(v2)$) {$1$};
\node at (0.65,-0.95) {\small $\sqrt{2}$};
\node at (1.38,0.45) {\small $\sqrt{6}/2$};
\node at (-0.65,-0.9) {\small $\sqrt{6}/2$};
\node[above=1.5pt] at ($(v2)!0.5!(v4)$) {\small $\sqrt{2}/2$};
\node[left=2pt] at ($(v2)!0.5!(v5)$) {\small $\sqrt{2}/2$};
\node[right=2pt] at ($(v3)!0.5!(v4)$) {$1$};
\node[below=2pt] at ($(v3)!0.5!(v5)$) {$1$};
\end{tikzpicture}
\caption{The graph in Example \ref{example}}
\label{fig:example-graph}
\end{figure}
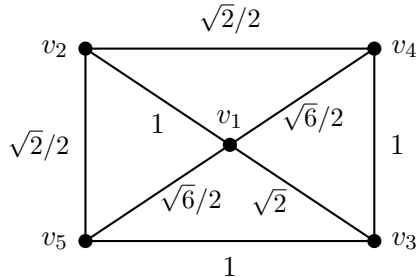
According to Theorem \ref{thm:main}, the spectral radius of $G$ is at most $2\sqrt{3}$. 
We will demonstrate that this upper bound is actually achieved, proving $\lambda (G) = 2\sqrt{3}$.
Set $\bm{w} = \sqrt[4]{3} (1, \sqrt{3}/3, \sqrt{6}/3, \sqrt{2}/2, \sqrt{2}/2)^{\mathrm{T}}$. 
One can directly verify that $\bm{w}$ satisfies item (1). Moreover, $\|\bm{w}\|^2 = 3\sqrt{3}$, 
$\|A(G)\|_F = 3\sqrt{2}$, and $\|\bm{1}_{V_i}\circ\bm{w}\|^2 = \sqrt{3} = \|\bm{w}\|^2 - \sqrt{2/3} \cdot \|A(G)\|_F$ 
for $i\in [3]$. Thus, by Theorem \ref{thm:main}, we conclude that $\lambda (G) = 2\sqrt{3}$.
\end{example}

The original motivation of the above inequality is the following problem.

\begin{problem}\label{problem-1}
Is there an extremal result on weighted graphs which implies Brada\v{c}'s Theorem 
\textup{(}i.e., a localized version of Tur\'an Theorem, see Corollary \ref{coro:Bradc2022}\textup{)}?
\end{problem}

A weaker result regarding Problem \ref{problem-1} can be found in \cite{LiuN26}.

\begin{corollary}[\cite{LiuN26}]\label{coro:Liu-Ning-26}
Let $G$ be a graph with spectral radius $\lambda(G)$. Then 
\begin{equation}\label{eq:local-bound}
\lambda(G) \leq \sqrt{2 \sum_{e\in E(G)} \frac{\cl(e) - 1}{\cl(e)}}.
\end{equation}
Equality holds if and only if $G$ is a complete bipartite graph for $\omega = 2$, or a 
complete regular $\omega$-partite graph for $\omega \geq 3$ \textup{(}possibly with some isolated vertices\textup{)}.
\end{corollary}

We emphasize that extending the above result to weighted graphs is essential, 
as it allows our new result to encompass Brada\v{c}'s extremal theorem and some of 
the following results --- something our previous result (Corollary \ref{coro:Liu-Ning-26}) 
could not achieve.

There are also similar phenomenons supporting this idea. For example, Li and Ning \cite{LiN25+} 
showed that the classic Frieze-McDiarmid-Reed Theorem \cite{FMR1992} on paths  and Bondy-Fan Theorem \cite{BF91} 
on cycles in weighted graphs imply Malec-Tompkins Theorem \cite{MT23} on paths, 
and Zhao-Zhang Theorem \cite{ZZ25} on cycles in terms of localized versions, and surely derive shorts proofs.

Theorem \ref{thm:main} unifies a series of previous results on upper bounds of spectral radius and new ones, 
and also some results from extremal graph theory. It should be mentioned that Nikiforov's inequality 
(Corollary \ref{coro:Niki-spectral-Turan-edge}) can imply Stanley's inequality (Corollary \ref{coro:Stanley}) 
via Tur\'an's theorem (see \cite{Nikiforov20}), our result offers a broader scope. 
Specifically, our inequality implies Hong's inequality (Corollary \ref{coro:Hong}) 
and some new spectral inequalities, and leads to some new theorems in extremal graph theory. 
For example, we derive the following theorem:

\begin{theorem}\label{thm:lowerbound-c(e)}
Let $G$ be a graph on $n\geq 2$ vertices with no isolated vertices. Then
\begin{equation}\label{eq:sum_ce_minor_lower_bound}
\sum_{e\in E(G)} \frac{1}{\cl(e) - 1} \geq \frac{n}{2},
\end{equation} 
and
\begin{equation}\label{eq:sum_ce_lower_bound}
\sum_{e\in E(G)} \frac{1}{\cl(e)} \geq \frac{n-c(G)}{2}, 
\end{equation} 
where $c(G)$ is the number of connected components of $G$.
\end{theorem}

Answering an open problem proposed by Balogh and Lidicky at the Combinatorics, 
Probability and Computing conference held in Oberwolfach in 2022,  
Brada\v{c} \cite{{Bradac2022}} established the inequality 
\[
\sum_{e \in E(G)} \frac{1}{\cl(e)-1} \leq \frac{n^2}{2} - m.
\]
Malec and Tompkins \cite{MT23} proved this result by induction method, independently.
Our result above complements with the upper bound established independently by Brada\v{c} \cite{Bradac2022}, 
and Malec and Tompkins \cite{MT23}.

For a graph $G$ and a vertex $v \in V(G)$, let $\cl_G(v)$ (or simply $\cl(v)$) denote the size of a 
largest clique containing $v$. We also highlight the following result, which provides 
a new generalization of Tur\'an's theorem based on vertex degree and the quantity $\cl(v)$.

\begin{theorem}\label{thm:local-Turan-degree}
Let $G$ be a graph on $n$ vertices. Then
\begin{equation}\label{eq:local-Turan-degree}
\sum_{v\in V(G)} d(v)\cdot\frac{\cl(v) - 1}{\cl(v)}
\leq \Bigg(\sum_{v\in V(G)} \frac{\cl(v) - 1}{\cl(v)}\Bigg)^2. 
\end{equation}
Equality holds if and only if $G$ is, up to isolated vertices, a complete regular
multi-partite graph.
\end{theorem}

\begin{remark}
Observe that the difference between the left-hand side and the right-hand side of \eqref{eq:local-Turan-degree} 
decreases monotonically with respect to each $\cl(v)$. Therefore, if $G$ is $K_{r+1}$-free, then $\cl(v) \leq r$ 
for all $v \in V(G)$, which implies $e(G) \leq (1 - 1/r) n^2/2$. 
\end{remark}

We thus arrive at the following theorem, which provides a local refinement of Wilf’s theorem \cite{Wilf1986} 
and follows from Corollary \ref{coro:Liu-Ning-26} and Theorem \ref{thm:local-Turan-degree}.

\begin{theorem}[Localized version of Wilf's theorem]\label{thm:spectral-localized-Turan-vertex}
Let $G$ be a graph with spectral radius $\lambda(G)$. Then
\[
\lambda(G)\leq \sum_{v\in V(G)}\frac{\cl(v)-1}{\cl(v)}.
\]
Equality holds if and only if $G$ is, up to isolated vertices, a regular complete multi-partite graph.
\end{theorem}

Theorem \ref{thm:spectral-localized-Turan-vertex} confirms a conjecture proposed by Tripathi \cite{T25} 
independently of us. In fact, this conjecture originated from research in Probability Theory and 
Operator Algebras, and Tripathi \cite{T25} told us that our result have applications to Khintchine-type inequalities, 
specifically in bounding the operator norms of sums for mixtures
of free and tensor-independent semicircle random variables.

Recent progress on Khintchine-type inequalities reveals a close connection between 
spectral graph theory, non-commutative probability, and combinatorial group theory. 
In particular, Collins and Miyagawa \cite{Collins2025} established operator-norm 
bounds for sums of semicircular random variables with mixed classical and free 
independence in terms of the largest eigenvalues of an associated graph, 
thereby providing analogs of a classic operator-valued Khintchine-type inequality 
obtained by Haagerup and Pisier \cite{Haagerup-Pisier1993}. Building on this, Santos, 
Tripathi and Youssef \cite{Santos-Tripathi-Youssef2025} addressed a more general setting 
based on $G$-independence, showing that operator-norm bounds can be expressed through combinatorial quantities $\cl(v)$,
and that extremal norm behavior can be formulated and resolved via Tur\'an-type 
problems. Collectively, these results highlight that the norm bounds of mixtures 
of free and classical independence can be effectively controlled by graph 
parameters\,--\,specifically the spectral radius, clique sizes, and structural 
extremizers. This suggests that spectral Tur\'an-type results may play a crucial 
role in understanding operator norms of sums of semicircular variables 
in intermediate regimes.

{\bfseries Notation.} For a graph $G$ of order $n$, let $\chi(G)$ and $\omega(G)$ 
denote its chromatic number and clique number, abbreviated as $\chi$ and $\omega$, 
respectively. For a subset $X\subseteq V(G)$, let $G[X]$ be the subgraph of $G$ 
induced by $X$, and $e(X)$ be the number of edges in $G[X]$. We also use $e(G)$ 
to denote the number of edges of $G$. For a vertex $v$ of $G$, we use $d_G(v)$ 
and $N_G(v)$ to denote the degree of $v$ and the set of neighbors of $v$ in $G$, 
respectively; we simplify these to $d(v)$ and $N(v)$ when the context is clear. 
For graph notation and terminology not defined here, we refer to \cite{Bondy-Murty2008}.

{\bfseries Outline.} We prove Theorem \ref{thm:main} in Section \ref{sec:proof-main}, 
relying on a new weighted extension of the Motzkin--Straus 
Theorem with weights based on $\cl(v)$. Subsequently, we provide the proofs of 
Corollary \ref{coro:Liu-Ning-26} and Corollary \ref{thm:AC}\,--\,Corollary \ref{coro:weighted-Cvetkovic}, 
as well as the proof of Theorem \ref{thm:local-Turan-degree} 
in Section \ref{sec:proofs}. Finally, Section \ref{sec:proofs-technical-results} is devoted 
to two technical results; a key one, Theorem \ref{thm:lowerbound-c(e)}, is used to show 
that our inequality implies Hong's inequality.

\section{Consequences}\label{Sec:2}
What's the meaning of a localized version of extremal theorem in Corollary \ref{coro:Liu-Ning-26}? 
In extremal graph theory, a central problem is to determine the maximum number of edges, denoted 
by $\ex(n, H)$, in an $n$-vertex graph that does not contain a forbidden subgraph $H$. A classical 
result in this area is Tur\'an's theorem, which resolves the case when $H=K_{r+1}$, showing that 
the unique extremal graph is the $r$-partite Tur\'an graph. This imposes a \emph{global} 
constraint, restricting the total number of edges. A more refined approach considers \emph{local} 
constraints. Instead of forbidding the presence of a $K_{r+1}$ globally, one may study local 
conditions, such as bounding a function defined over the edges --- for example, 
$\sum_{e \in E(G)} \frac{\cl(e)}{\cl(e)-1}$. This leads to a \emph{localized} version of the 
Tur\'an problem, offering a finer perspective on extremal graph properties.

Recall the following result due to Brada\v{c} \cite{Bradac2022}, 
and Malec and Tompkins \cite{MT23}, independently.
Applying Theorem \ref{thm:main} allows for a very concise proof of Corollary \ref{coro:Bradc2022}, 
which differs from the inductive proof in \cite{MT23}.

\begin{corollary}[\cite{Bradac2022,MT23}]\label{coro:Bradc2022}
Let $G$ be any $n$-vertex graph. Then
\[
\sum_{e\in E(G)}\frac{\cl(e)}{\cl(e)-1}\leq \frac{n^2}{2}.
\]
\end{corollary}

\begin{proof}
For the graph $G$, we assign to each edge $e\in E(G)$ a weight $w(e) = \frac{\cl(e)}{\cl(e)-1}$, 
and let $G_w$ denote the resulting weighted graph. The total edge weight of $G_w$ is then 
\[
W := \sum_{e\in E(G)} \frac{\cl(e)}{\cl(e)-1}.
\]
By Rayleigh's inequality, $\lambda (G_w)\geq 2W/n$. Combining this with Theorem \eqref{thm:main} 
and $w(e) = \frac{\cl(e)}{\cl(e) - 1}$, we obtain
\[
\frac{2}{n} \sum_{e\in E(G)} \frac{\cl(e)}{\cl(e)-1} = \frac{2W}{n} 
\leq\lambda(G_w)\leq\sqrt{2 \sum_{e\in E(G)} \frac{\cl(e) - 1}{\cl(e)} w(e)^2}.
\]
Solving the above inequality yields the desired result.
\end{proof}

Brada\v{c} \cite{Bradac2022} was the first to notice the relationship between the Motzkin-Straus 
result and localized version of Tur\'an-type problem. This line of inquiry was 
later developed systematically by Malec and Tompkins \cite{MT23}, who explored 
localized versions of several extremal problems, including the localized version of 
Erd\H{o}s-Gallai Theorem on paths, and Localized EKR Theorem. 
Zhao and Zhang \cite{ZZ25} established a localized Erd\H{o}s-Gallai Theorem on cycles. In addition, they
proved a localized version of the Erd\H{o}s-Gallai theorem for hypergraphs \cite{ZZ25-2}. Very recently, Li and Ning \cite{LiN25+} 
proved localized version of Erd\H{o}s-Gallai Theorem on matchings and Balister-Bollob\'as-Riordan-Schelp  
Theorem \cite{BBRS2003} on paths.

Another consequence due to Adak and Chandran is the following. We present 
a new and short proof in Section \ref{sec:proofs}.

\begin{corollary}[Adak and Chandran \cite{AC-25arxiv}]\label{thm:AC}
Let $G$ be a graph on $n$ vertices and $m$ edges. Then
\[
m\leq \frac{n}{2} \sum_{v\in V(G)}\frac{\cl(v)-1}{\cl(v)}.
\]
Equality holds if and only if $G$ is either a regular
complete multi-partite graph or an empty graph.
\end{corollary}


Our main results has several applications in spectral graph theory. Recall that
in 1976, Brualdi and Hoffman posed a classical problem concerning the maximum eigenvalue of a 
$(0,1)$-matrix with a prescribed number of edges (see \cite[pp.~438]{BFLS1978}). Partial progress 
was made by Friedland \cite{Friedlan1985} and Stanley \cite{Stanley1987}, and the problem was 
later completely resolved by Rowlinson \cite{Rowlinson1988}. A detailed discussion of these 
foundational contributions can be found in a recent paper \cite{Cheng-Weng2025}. Of particular note, 
Stanley's method led to the following elegant spectral inequality.

\begin{corollary}[{Stanley \cite[pp. 268]{Stanley1987}}]\label{coro:Stanley}
Let $G$ be a graph on $m$ edges. Then
\[
\lambda(G)\leq -\frac{1}{2}+\sqrt{2m+\frac{1}{4}}.
\]   
\end{corollary}

Among the extensions of Stanley's inequality is the following result by Hong,
which implies a non-trivial upper bound of spectral radius of planar graphs.

\begin{corollary}[{Hong \cite[pp.~69]{Hong1993}}]\label{coro:Hong}
Let $G$ be a graph on $n$ vertices and $m$ edges. If $\delta(G)\geq 1$ then
\[
\lambda(G)\leq \sqrt{2m-n+1}.
\]    
\end{corollary}

By imposing minimum degree as a new parameter, Hong's inequality was further improved independently 
by Hong, Shu, and Fang \cite{HongShuFang2001}, and by Nikiforov \cite{Nikiforov2002}. Specifically, 
if $G$ is a connected graph on $n$ vertices and $m$ edges with minimum degree $\delta (G)\geq 1$, then 
\[
\lambda(G)\leq \frac{\delta (G) - 1}{2}+\sqrt{2m - \delta (G) n + \frac{(\delta (G) + 1)^2}{4}}.
\]
In another direction, confirming a conjecture of Edwards and Elphick \cite{EE1983}, and improving 
two classic spectral inequalities of Wilf \cite{Wilf1967,Wilf1986}, Nikiforov \cite{Nikiforov2002} 
proved the following spectral inequality.

\begin{corollary}[{Nikiforov \cite[Theorem 2.1]{Nikiforov2002}}]\label{coro:Niki-spectral-Turan-edge}
Let $G$ be a graph on $m$ edges. If $G$ is $K_{r+1}$-free, then 
\[
\lambda(G)\leq \sqrt{2\Big(1 - \frac{1}{r}\Big) m}.
\]
\end{corollary}

Only very recently, the authors \cite{LiuN26} proved a localized version of spectral 
Tur\'an theorem (Corollary \ref{coro:Liu-Ning-26}). In this paper, we also present four 
new results in this direction.

\begin{corollary}\label{coro:spectral-vertex-degree-localized}
Let $G$ be a graph with spectral radius $\lambda(G)$. Then
\begin{equation}\label{eq:local-bound-vertex}
\lambda(G) \leq \sqrt{\sum_{v\in V(G)} d(v)\cdot \frac{\cl(v)-1}{\cl(v)}}.
\end{equation}
Equality holds if and only if $G$ is, up to isolated vertices, a complete bipartite graph 
for $\omega(G) = 2$, or a regular complete $\omega(G)$-partite graph for $\omega(G) \geq 3$.
\end{corollary}

Since $\cl(v) \leq n$ holds for each vertex $v \in V(G)$, Wilf's inequality, 
$\lambda(G) \leq \sqrt{2(1 - 1/n)m}$ (see \cite[Corollary,~pp.~331]{Wilf1967}), is an immediate 
corollary of inequality \eqref{eq:local-bound-vertex}. This connection 
motivates our next result, a local refinement of another inequality due to Wilf \cite{Wilf1986}, 
which is formally stated as Theorem \ref{thm:spectral-localized-Turan-vertex} in Section \ref{sec:1}.

\begin{corollary}[Localized version of Wilf's theorem]\label{coro:spectral-localized-Turan-vertex}
Let $G$ be a graph with spectral radius $\lambda(G)$. Then
\begin{align*}
\lambda(G)\leq \sum_{v\in V(G)}\frac{\cl(v)-1}{\cl(v)}.
\end{align*}
Equality holds if and only if $G$ is, up to isolated vertices, a regular complete multi-partite graph.
\end{corollary}

As a direct consequence of Corollary \ref{coro:spectral-localized-Turan-vertex}, we recover 
another theorem of Wilf \cite{Wilf1986}, which states that $\lambda(G)\leq (1 - 1/\omega(G))n$ for any graph $G$.

\begin{corollary}[Weighted Edwards-Elphick Theorem]\label{coro:weighted-Edwards-Elphick}
Let $G$ be a weighted graph with spectral radius $\lambda(G)$. Then
\begin{equation}\label{eq:local-bound}
\chi(G)\geq 1+\frac{\lambda(G)^2}{\|A(G)\|_F^2-\lambda(G)^2}.
\end{equation}    
\end{corollary}

The above result provides a weighted generalization of the Edwards-Elphick Theorem \cite{EE1983}.

\begin{corollary}[Weighted Cvetkovi\'c Theorem]\label{coro:weighted-Cvetkovic}
Let $G$ be a weighted graph with an edge weight $w(e)$ for each edge $e$ and spectral radius $\lambda(G)$. 
Suppose that $2\sum_{e\in E(G)} w(e)\geq \|A(G)\|_F^2$. Then
\begin{equation}\label{eq:local-bound}
\chi(G)\geq 1+\frac{\lambda(G)}{n-\lambda(G)}.
\end{equation}    
\end{corollary}

This generalizes a theorem of Cvetkovi\'c \cite{Cvetkovic1972}, which corresponds to 
the case where all edge weights are $1$.

\section{A new localized spectral Tur\'an theorem for weighted graphs}
\label{sec:proof-main}

The goal of this section is to present a proof of Theorem \ref{thm:main}.
Before the proof, we need some additional notation. Let $\triangle^{n-1}$ be the standard simplex:
\[
\triangle^{n-1}: = \bigg\{(x_1,x_2,\ldots,x_n) \in\mathbb{R}^n : x_i \geq 0, i\in [n] \ \text{and}\ \sum_{i=1}^n x_i = 1\bigg\}.
\]
Let $G$ be a graph on $n$ vertices with clique number $\omega (G)$. The well-known Motzkin--Straus
Theorem \cite{Motzkin-Straus1965} states that
\begin{equation}\label{eq:Motzkin-Straus-inequality}
\max\{\bm{z}^{\mathrm{T}} A(G) \bm{z}: \bm{z}\in\triangle^{n-1}\} = 1 - \frac{1}{\omega (G)}.
\end{equation}
The above equation established a remarkable connection between the clique number and the Lagrangian of a graph.

Additionally, Motzkin and Straus also gave a characterization on equality $\bm{z}^{\mathrm{T}} A(G) \bm{z} = 1-1/\omega(G)$. 
Given a vector $\bm{x}$, the \emph{support} of $\bm{x}$, denoted by $\supp (\bm{x})$, is the set consisting of all indices
corresponding to nonzero entries in $\bm{x}$. 

\begin{proposition}[\cite{Motzkin-Straus1965}]\label{prop:MS-equality}
Let $G$ be an $n$-vertex graph with $\omega(G)=\omega$, and let $\bm{x}\in\triangle^{n-1}$.
Then $\bm{x}^{\mathrm{T}} A(G) \bm{x} = 1-1/\omega$ holds if and only if the induced subgraph of $G$ on
$\supp (\bm{x})$ is a complete $\omega$-partite graph whose vertex classes $V_1,V_2,\ldots,V_{\omega}$ 
satisfy $\sum_{v\in V_i} x_v = 1/\omega$ for all $i\in [\omega]$.
\end{proposition}

We will use the following lemma. Its proof follows an argument similar to one in \cite{LiuN26}.

\begin{lemma}\label{lem:pre-for-weighted-MS}
Let $G$ be an $n$-vertex graph, and $W\in\mathbb{R}^{n\times n}$ be a nonnegative symmetric matrix. 
If $\bm{x}\in\triangle^{n-1}$ is a vector that attains $\max\{\bm{z}^{\mathrm{T}} (W\circ A(G)) \bm{z}: \bm{z}\in \triangle^{n-1}\}$, 
then 
\begin{enumerate}
\item[$(1)$] $\bm{x}^{\mathrm{T}} (W\circ A(G))\bm{e}_i = \bm{x}^{\mathrm{T}} (W\circ A(G))\bm{e}_j$ 
for any $i,j\in\supp (\bm{x})$. Here $\bm{e}_i$ is the $i$-th column of the identity matrix of order $n$.

\item[$(2)$] If furthermore, $\bm{x}$ has minimal support set, then $\supp (\bm{x})$ induces a clique in $G$.
\end{enumerate}
\end{lemma}

\begin{proof}
For brevity, we denote $A:= A(G)$ and $F_G(\bm{z}) := \bm{z}^{\mathrm{T}} \big(W\circ A\big) \bm{z}$ 
for $\bm{z}\in\mathbb{R}^n$. Assume that $\bm{x}\in\triangle^{n-1}$ is a vector that attains the maximum of
$F_G(\,\cdot\,)$ over $\triangle^{n-1}$.

(1) Let $i,j\in\supp (\bm{x})$. We assume, without loss of generality, towards contradiction that
$\bm{x}^{\mathrm{T}} W\bm{e}_i > \bm{x}^{\mathrm{T}} (W\circ A)\bm{e}_j$.
Let $0<\varepsilon<\min\{x_i, x_j\}$. Denote $\bm{y} := \bm{x} + \varepsilon (\bm{e}_i - \bm{e}_j)$.
As $W\circ A$ has zero diagonal, $\bm{e}_i^{\mathrm{T}} (W\circ A) \bm{e}_i=0$. Then
\begin{align*}
F_G(\bm{y}) - F_G(\bm{x})
& = 2\varepsilon \bm{x}^{\mathrm{T}} (W\circ A)(\bm{e}_i - \bm{e}_j) 
+ \varepsilon^2 (\bm{e}_i-\bm{e}_j)^{\mathrm{T}} (W\circ A) (\bm{e}_i - \bm{e}_j) \\
& = 2\varepsilon \bm{x}^{\mathrm{T}} (W\circ A)(\bm{e}_i - \bm{e}_j) - 2\varepsilon^2 \bm{e}_i^{\mathrm{T}} (W\circ A) \bm{e}_j.
\end{align*}
Choosing $\varepsilon$ to be sufficiently small gives $F_G(\bm{y}) > F_G(\bm{x})$,
a contradiction.

(2) Let $\bm{x}$ have minimal support set. Assume for contradiction that there exist 
distinct $i,j\in \supp(\bm{x})$ such that $ij\notin E(G)$. Without loss of generality, 
assume $\bm{x}^{\mathrm{T}} (W\circ A) \bm{e}_i\geq \bm{x}^{\mathrm{T}} (W\circ A) \bm{e}_j$.

Consider the vector $\bm{y} = \bm{x} + x_j(\bm{e}_i - \bm{e}_j)$, and set
$\bm{e}:= \bm{e}_i - \bm{e}_j$. A direct calculation shows that $\bm{y}\in\triangle^{n-1}$, and
\begin{align*}
F_G(\bm{y}) - F_G(\bm{x}) 
& = \bm{y}^{\mathrm{T}} (W\circ A) \bm{y} - \bm{x}^{\mathrm{T}} (W\circ A) \bm{x} \\
& = 2x_j \bm{x}^{\mathrm{T}} (W\circ A) \bm{e} + x_j^2 \bm{e}^{\mathrm{T}} (W\circ A) \bm{e}.
\end{align*}
Since $ij\notin E(G)$, we have $\bm{e}^{\mathrm{T}} (W\circ A) \bm{e} = 0$. Consequently,
\[
F_G(\bm{y}) - F_G(\bm{x}) = 2x_j\bm{x}^{\mathrm{T}} (W\circ A) \bm{e} = 2x_j \bm{x}^{\mathrm{T}} (W\circ A) (\bm{e}_i - \bm{e}_j) \geq 0.
\]
This implies that $\bm{y}$ also attains the maximum of $F_G(\,\cdot\,)$ over $\triangle^{n-1}$.
However, $\supp (\bm{y}) < \supp (\bm{x})$, which contradicts the minimality of the support of $\bm{x}$.
So we finish the proof of item (2).
\end{proof}

For a graph $G$ on $n$ vertices and no isolated vertices, we define the $n\times n$ symmetric matrix 
$W(G) = [w_{ij}]$ by 
\[
w_{ij} = \frac{1}{2} \bigg(\frac{\cl(i)}{\cl(i)-1} + \frac{\cl(j)}{\cl(j)-1}\bigg).
\]

\begin{lemma}\label{lem:generalized-MS-inequality}
Let $G$ be an $n$-vertex graph with $\omega(G)=\omega$ and no isolated vertices, and 
let $\bm{x}\in\triangle^{n-1}$. Then
\[
\bm{x}^{\mathrm{T}} \big(W(G)\circ A(G)\big) \bm{x} \leq 1.
\]
Equality holds if and only if the induced subgraph of $G$ on $\supp (\bm{x})$ is a complete $\omega$-partite graph 
whose vertex classes $V_1,V_2,\ldots,V_{\omega}$ satisfy $\sum_{v\in V_i} x_v = 1/\omega$ for all $i\in [\omega]$.
\end{lemma}

\begin{proof}
For brevity, we denote $W:= W(G)$, $A:= A(G)$ and $F_G(\bm{z}) := \bm{z}^{\mathrm{T}} (W\circ A) \bm{z}$. 
We first assume $\bm{z}\in\triangle^{n-1}$ is a vector that attains the maximum of
$F_G(\,\cdot\,)$ over $\triangle^{n-1}$, and has minimal support set.

By Lemma \ref{lem:pre-for-weighted-MS}, $\supp (\bm{z})$ forms a clique $K$ in $G$. Hence,
\[
F_G(\bm{z}) = \bm{z}^{\mathrm{T}} (W\circ A) \bm{z} = \sum_{ij\in E(K)} 
\bigg(\frac{\cl(i)}{\cl(i)-1} + \frac{\cl(j)}{\cl(j)-1}\bigg) z_iz_j.
\]
Observing that $x/(x-1)$ is a decreasing function in $x$, we obtain
\[
F_G(\bm{z}) \leq \frac{2|K|}{|K| - 1} \sum_{ij\in E(K)} z_iz_j \leq 1,
\]
where the last inequality follows by applying \eqref{eq:Motzkin-Straus-inequality} to the clique $K$.

In the following we characterize the equality for $F_G(\bm{x}) = 1$. If $G[\supp (\bm{x})]$ is a 
complete $\omega$-partite graph whose vertex classes $V_1,V_2,\ldots,V_{\omega}$ satisfy 
$\sum_{v\in V_i} x_v = 1/\omega$ for all $i\in [\omega]$, then $\bm{x}^{\mathrm{T}} A(G)\bm{x} = 1-1/\omega$
by Proposition \ref{prop:MS-equality}. Since $\cl(u) \leq\omega$ for each $u\in V(G)$, we have
\begin{align*}
F_G(\bm{x}) & = \sum_{ij\in E(G)} 
\bigg(\frac{\cl(i)}{\cl(i)-1} + \frac{\cl(j)}{\cl(j)-1}\bigg) x_ix_j \\
& \geq \frac{2\omega}{\omega-1} \sum_{ij\in E(G)} x_ix_j \\
& = \frac{\omega}{\omega-1} \cdot \bm{x}^{\mathrm{T}} A(G)\bm{x} = 1,
\end{align*}
which yields that $F_G(\bm{x})=1$.

For the other direction, assume that $F_G(\bm{x})=1$. We proceed by induction on $n$ to demonstrate that
$G[\supp (\bm{x})]$ is a complete $\omega$-partite graph with vertex classes $V_1,V_2,\ldots,V_{\omega}$, 
and $\sum_{v\in V_i} x_v = 1/\omega$ for all $i\in [\omega]$. The base case $n=2$ is clear. Now assume $n\geq 3$. 
If $G[\supp(\bm{x})]$ is a complete graph, the result holds trivially. Otherwise, there exist 
$v_1,v_2\in\supp(\bm{x})$ such that $v_1v_2\notin E(G)$. Let $G':= G-v_1$, and define a vector 
$\bm{x}'\in\mathbb{R}^{n-1}$ for $G'$ by
\[
x'_v =
\begin{cases}
x_v, & v\neq v_2, \\
x_{v_1} + x_{v_2}, & v = v_2.
\end{cases}
\]
From Lemma \ref{lem:pre-for-weighted-MS}, we have 
$\bm{x}^{\mathrm{T}} (W\circ A)\bm{e}_{v_1} = \bm{x}^{\mathrm{T}} (W\circ A)\bm{e}_{v_2}$, which implies
\begin{equation}\label{eq:auxiliary-eq-4}
\sum_{v\in N_G(v_1)} \bigg(\frac{\cl(v)}{\cl(v)-1} + \frac{\cl(v_1)}{\cl(v_1)-1}\bigg) x_v 
= \sum_{v\in N_G(v_2)} \bigg(\frac{\cl(v)}{\cl(v)-1} + \frac{\cl(v_2)}{\cl(v_2)-1}\bigg) x_v. 
\end{equation}
Since $v_1v_2\notin E(G)$, for each vertex $v\in V(G')$ adjacent to $v_2$, the value $\cl(v)$ 
remains the same in $G$ and $G'$. Combining this with \eqref{eq:auxiliary-eq-4} gives
\begin{align*}
& ~ \sum_{vv_2\in E(G')} \bigg(\frac{\cl_{G'}(v)}{\cl_{G'}(v)-1} + \frac{\cl_{G'}(v_2)}{\cl_{G'}(v_2)-1}\bigg) x'_vx'_{v_2} \\
= & ~ \sum_{vv_2\in E(G)} \bigg(\frac{\cl_G(v)}{\cl_G(v)-1} + \frac{\cl_G(v_2)}{\cl_G(v_2)-1}\bigg) (x_{v_1} + x_{v_2}) x_v \\
= & ~ \sum_{vv_2\in E(G)} \bigg(\frac{\cl(v)}{\cl(v)-1} {+} \frac{\cl(v_2)}{\cl(v_2)-1}\bigg) x_{v_2} x_v 
+ \sum_{vv_1\in E(G)} \bigg(\frac{\cl(v)}{\cl(v)-1} {+} \frac{\cl(v_1)}{\cl(v_1)-1}\bigg) x_{v_1} x_v.
\end{align*}
Consider any edge of $G'$ that is not incident to $v_2$. For such an edge $uv$, we have 
\[
\bigg(\frac{\cl_{G'}(u)}{\cl_{G'}(u)-1} + \frac{\cl_{G'}(v)}{\cl_{G'}(v)-1}\bigg) x'_ux'_v
\geq \bigg(\frac{\cl_G(u)}{\cl_G(u)-1} + \frac{\cl_G(v)}{\cl_G(v)-1}\bigg) x_ux_v.
\]
Hence, $F_{G'}(\bm{x}') = 1$. By the induction hypothesis, the induced subgraph $H$ of $G'$ on 
$\supp (\bm{x}')$ is a complete $k$-partite graph, where
$k=\omega(G')$. Let $V_1,\ldots,V_k$ be the vertex classes of $H$. 

{\noindent\bfseries Case 1.} $N_G(v_1)\cap V_i \neq\emptyset$ for all $i\leq [k]$. 

In this case, $\omega(G[\supp(\bm{x})]) = k+1$. If there exists some $j$ 
such that $|N_G(v_1)\cap V_j| < |V_j|$, then one can find a vertex $w\in G[\supp(\bm{x})]$ with 
$\cl(w) \leq k$ in $G[\supp(\bm{x})]$. This would imply $F_G(\bm{x}) > 1$, a contradiction. 
Hence, $v_1$ is adjacent to every vertex in $V_1\cup\cdots\cup V_k$. Hence, $G[\supp(\bm{x})]$ 
is a complete $(k+1)$-partite graph, as desired. 

{\noindent\bfseries Case 2.} There exists some $V_i$, say $V_1$, such that $v_1$ is nonadjacent 
to any vertices in $V_1$. 

Since $F_G(\bm{x}) = 1$ is maximal, Lemma \ref{lem:pre-for-weighted-MS} implies that for each $v\in V_1$,
we have $\bm{z}^{\mathrm{T}} (W\circ A)\bm{e}_{v_1} = \bm{z}^{\mathrm{T}} (W\circ A)\bm{e}_v$.
It follows that $v_1$ is adjacent to all vertices in $V_2\cup\cdots\cup V_k$. Furthermore, 
since $v_1v_2\notin E(G)$, it must be that $v_2\in V_1$. Thus, $V_1\cup \{v_1\}, V_2,\ldots,V_k$ 
form a complete $\omega (G)$-partite graph, and $\sum_{v\in V_1\cup\{v_1\}} x_v = \sum_{v\in V_i} x_v$ 
for $i=2,\ldots,\omega (G)$, completing the proof.
\end{proof}

For any edge $e=uv\in E(G)$, it is clearly $\cl(e)\leq\min\{\cl(u), \cl(v)\}$. Hence we immediately 
have the following result.

\begin{corollary}\label{coro:weighted-MS-inequality}
Let $G$ be an $n$-vertex graph with $\omega(G)=\omega$ and $\bm{x}\in\triangle^{n-1}$. Then
\[
2\sum_{ij\in E(G)} \frac{\cl(ij)}{\cl(ij)-1} x_ix_j \leq 1.
\]
Equality holds if and only if the induced subgraph of $G$ on $\supp (\bm{x})$ is a complete $\omega$-partite graph 
whose vertex classes $V_1,V_2,\ldots,V_{\omega}$ satisfy $\sum_{v\in V_i} x_v = 1/\omega$ for all $i\in [\omega]$.
\end{corollary}

Now we are in standing for presenting the proof of our main theorem.

\begin{proof}[Proof of Theorem \ref{thm:main}] 
Let $A(G) = [w_{ij}]$ be the adjacency matrix of $G$, where $w_{ij} = w(e)$ if $e=ij\in E(G)$. 
Let $\bm{x}$ be a unit eigenvector of $A(G)$ corresponding to some eigenvalue of $A(G)$,
whose absolute value is equal to $\lambda (G)$, and define the vector $\bm{y}$ by $y_i = x_i^2$ 
for $i\in [n]$. Then we obtain
\begin{align*}
\lambda (G)
& = 2 \,\bigg|\sum_{ij\in E(G)} w_{ij} x_ix_j\bigg| \leq 2\cdot \sum_{ij\in E(G)} |w_{ij}| \sqrt{y_iy_j} \\
& = 2 \sum_{e\,=\, ij\in E(G)} |w_{ij}|\sqrt{\frac{\cl(e) - 1}{\cl(e)}} \cdot \sqrt{\frac{\cl(e)}{\cl(e) - 1} y_iy_j}.
\end{align*}
Using Cauchy--Schwarz inequality and Corollary \ref{coro:weighted-MS-inequality}, we find
\begin{align*}
\lambda (G)^2
& \leq 4 \Bigg(\sum_{e\in E(G)} \frac{\cl(e) - 1}{\cl(e)} w(e)^2\Bigg) 
\Bigg(\sum_{e\,=\, ij\in E(G)} \frac{\cl(e)}{\cl(e) - 1} y_iy_j\Bigg) \\
& \leq 2 \sum_{e\in E(G)} \frac{\cl(e) - 1}{\cl(e)} w(e)^2,
\end{align*}
completing the proof of \eqref{eq:weighted-local-bound}.

Next, we characterize all graphs attaining equality in \eqref{eq:weighted-local-bound}. 
From the proof above, equality in \eqref{eq:weighted-local-bound} holds if and only if 
the weight $w_{ij}$ has the same sign for each $ij\in E(G)$ for which $y_iy_j\neq 0$, and 
\[
\sum_{e\,=\, ij\in E(G)} \frac{\cl(e)}{\cl(e) - 1} y_iy_j 
= 1\quad \text{and}\quad \frac{\sqrt{y_iy_j}}{w_{ij}} \cdot\frac{\cl(e)}{\cl(e)-1}
\]
is a constant for each $e=ij\in E(G)$. By Corollary \ref{coro:weighted-MS-inequality} and $y_i > 0$ 
for any $i\in V(G)$ with $d(i) \geq 1$, the equality in \eqref{eq:weighted-local-bound} is 
equivalent to the following:

(i) $G$ is, up to isolated vertices, a complete $r$-partite graph whose vertex class $V_1,V_2,\ldots,V_r$ 
satisfy $\sum_{v\in V_i} y_v = 1/r$, $i\in [r]$;

(ii) For any $e=ij\in E(G)$, $w_{ij} = c\sqrt{y_iy_j}$ for some constant $c$.

We first show the necessity. Assume that equality holds in \eqref{eq:weighted-local-bound}. 
Our goal is to construct a vector $\bm{w}$ satisfying items (1) and (2).  

Without loss of generality, we assume $c>0$. Under this assumption, the adjacency matrix 
$A(G)$ is nonnegative, allowing us to take the eigenvector $\bm{x}$ to be nonnegative. 
Since $y_i = x_i^2$, it follows that $w_{ij} = cx_ix_j$ for each $ij\in E(G)$. We then 
determine the value of $c$ by equating two alternative expressions for $\lambda (G)$. 
From \eqref{eq:weighted-local-bound} we have
\begin{equation}\label{eq:equation-lambda}
\lambda (G) = \sqrt{2 \sum_{e\in E(G)} \frac{\cl(e) - 1}{\cl(e)} w(e)^2} 
= \sqrt{\frac{r-1}{r}} \cdot \|A(G)\|_F. 
\end{equation}
On the other hand, 
\[
\lambda (G) = 2\sum_{ij\in E(G)} w_{ij} x_ix_j = \frac{2}{c} \sum_{e\in E(G)} w(e)^2 
= \frac{\|A(G)\|_F^2}{c}.
\]
Combining the two equations for $\lambda (G)$, we deduce that
\begin{equation}\label{eq:c-value}
c = \sqrt{\frac{r}{r-1}}\cdot \|A(G)\|_F.
\end{equation}

Now define $\bm{w}:= \sqrt{c}\, (\bm{x}_{V_1}, \bm{x}_{V_2}, \ldots, \bm{x}_{V_r})^{\mathrm{T}}$. 
We will verify that this $\bm{w}$ satisfies items (1) and (2). 
For item (1), this follows from a direct verification, which we omit for brevity.
For item (2), we recall that $y_i = x_i^2$ for $i\in V(G)$ and $\sum_{v\in V_i} y_v = 1/r$ for each $i\in [r]$.
Consequently,
\begin{equation}\label{eq:sum-xu-square}
\|\bm{1}_{V_i}\circ \bm{x}\|^2 = \sum_{u\in V_i} x_u^2 = \frac{1}{r}.
\end{equation}
Thus, the left-hand side of the equation in item (2) is
\[
\|\bm{1}_{V_i} \circ \bm{w}\|^2 = c \|\bm{1}_{V_i} \circ \bm{x}\|^2
= \frac{c}{r} = \frac{\|A(G)\|_F}{\sqrt{r(r-1)}}.
\]
For the right-hand side, applying \eqref{eq:c-value} yields
\[
\|\bm{w}\|^2 - \sqrt{\frac{r-1}{r}} \cdot \|A(G)\|_F 
= c\cdot \|\bm{x}\|^2 - \sqrt{\frac{r-1}{r}} \cdot \|A(G)\|_F
= \frac{\|A(G)\|_F}{\sqrt{r(r-1)}}.
\]
It follows that $\|\bm{1}_{V_i}\circ\bm{w}\|^2 = \|\bm{w}\|^2 - \sqrt{1-1/r} \cdot \|A(G)\|_F$ 
for $i\in [r]$.

Finally, we consider the sufficiency for the equality in \eqref{eq:weighted-local-bound}. 
Suppose there exists a vector $\bm{w}$ such that the adjacency matrix decomposes as 
$A(G) = \sum_{i=1}^r (\bm{1}_{V_i}\circ\bm{w}) \big((\bm{1} - \bm{1}_{V_i})\circ\bm{w}\big)^{\mathrm{T}}$, 
and $\|\bm{1}_{V_i}\circ\bm{w}\|^2 = \|\bm{w}\|^2 - \sqrt{1-1/r} \cdot \|A(G)\|_F$ for $i\in [r]$.
(the case where $A(G) = -\sum_{i=1}^r (\bm{1}_{V_i}\circ\bm{w}) \big((\bm{1} - \bm{1}_{V_i})\circ\bm{w}\big)^{\mathrm{T}}$ follows from a similar argument.)
We begin by noting two key observations. First, for any $i, j\in [r]$, we have 
\[
\big((\bm{1} - \bm{1}_{V_i})\circ\bm{w}\big)^{\mathrm{T}} 
(\bm{1}_{V_j} \circ \bm{w}) = 
\begin{cases}
\|\bm{1}_{V_j}\circ\bm{w}\|^2, & i\neq j, \\
0, & i = j.
\end{cases}
\]
Second, the vector $\bm{w}$ can be decomposed over the partition as 
$\bm{w} = \sum_{j=1}^r (\bm{1}_{V_j} \circ \bm{w})$. 
Using these facts, we deduce that
\begin{align*}
A(G) \bm{w} 
& = \sum_{i=1}^r (\bm{1}_{V_i}\circ\bm{w}) \big((\bm{1} - \bm{1}_{V_i})\circ\bm{w}\big)^{\mathrm{T}} 
\Bigg(\sum_{j=1}^r (\bm{1}_{V_j} \circ \bm{w})\Bigg) \\
& = \sum_{i\neq j} (\bm{1}_{V_i}\circ\bm{w}) \big((\bm{1} - \bm{1}_{V_i})\circ\bm{w}\big)^{\mathrm{T}} 
(\bm{1}_{V_j} \circ \bm{w}) \\
& = \sum_{i\neq j} \|\bm{1}_{V_j}\circ \bm{w}\|^2 \cdot (\bm{1}_{V_i}\circ\bm{w}).
\end{align*}
This, along with careful calculations, indicates that
\begin{align*}
A(G) \bm{w} 
& = \sum_{i=1}^r \Bigg(\sum_{j:\, j\neq i} \|\bm{1}_{V_j}\circ \bm{w}\|^2\Bigg) (\bm{1}_{V_i}\circ\bm{w}) \\
& = \sum_{i=1}^r \big(\|\bm{w}\|^2 - \|\bm{1}_{V_i}\circ \bm{w}\|^2\big) (\bm{1}_{V_i}\circ\bm{w}) \\
& = \sqrt{\frac{r-1}{r}} \|A(G)\|_F \cdot \sum_{i=1}^r \bm{1}_{V_i}\circ\bm{w} \\
& = \sqrt{\frac{r-1}{r}} \|A(G)\|_F\, \bm{w}.
\end{align*}
Since $G$ is a complete $r$-partite graph, we have 
\[
\sqrt{2 \sum_{e\in E(G)} \frac{\cl(e) - 1}{\cl(e)} w(e)^2} = \sqrt{\frac{r-1}{r}} \|A(G)\|_F.
\]
It follows that equality holds in \eqref{eq:weighted-local-bound}.
This completes the proof of the theorem.
\end{proof}

\section{Proofs of consequences}
\label{sec:proofs}

This section is devoted to the proofs of Corollary \ref{coro:Liu-Ning-26}, 
Theorem \ref{thm:local-Turan-degree}, and Corollary \ref{thm:AC}\,--\,Corollary \ref{coro:weighted-Cvetkovic}.

To clarify the logical structure of the proofs, the following diagram outlines the 
dependencies among the results.

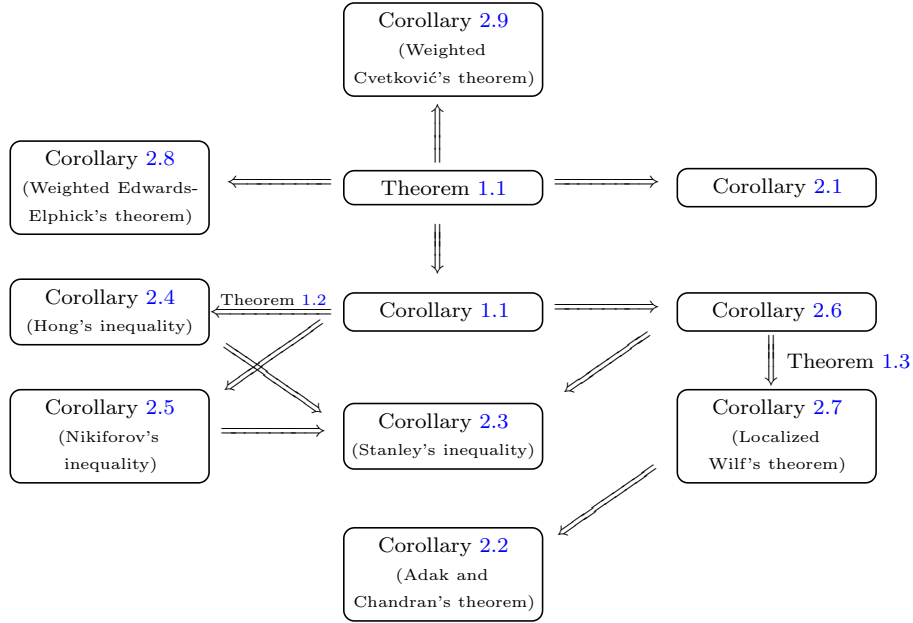
\begin{center}
\begin{tikzpicture}[every node/.style={font=\scriptsize}]
\matrix[row sep = 1.5em, column sep = 4.6em]{
& \node[chart] (weighted-Cvetkovic) {Corollary \ref{coro:weighted-Cvetkovic} \\ \tiny (Weighted Cvetkovi\'c's theorem)}; & \\
\node[chart] (weighted-Edwards-Elphick) {Corollary \ref{coro:weighted-Edwards-Elphick} \\ \tiny (Weighted Edwards-Elphick's theorem)}; 
& \node[chart] (thm) {Theorem \ref{thm:main}}; & \node[chart] (coro-2) {Corollary \ref{coro:Bradc2022}}; \\
\node[chart] (Hong) {Corollary \ref{coro:Hong} \\ \tiny (Hong's inequality)}; & \node[chart] (Liu-Ning) {Corollary \ref{coro:Liu-Ning-26}}; 
& \node[chart] (vertex-degree-localized) {Corollary \ref{coro:spectral-vertex-degree-localized}}; \\
\node[chart] (Niki) {Corollary \ref{coro:Niki-spectral-Turan-edge} \\ \tiny (Nikiforov's inequality)}; & \node[chart] (Stanley) {Corollary \ref{coro:Stanley} \\ \tiny (Stanley's inequality)}; 
& \node[chart] (localized-Wilf) {Corollary \ref{coro:spectral-localized-Turan-vertex} \\ \tiny (Localized Wilf's theorem)}; \\
& \node[chart] (Adak-Chandran) {Corollary \ref{thm:AC} \\ \tiny (Adak and Chandran's theorem)}; & \\};
\node[rotate=90] at ($(thm.north)!0.5!(weighted-Cvetkovic.south)$) {$\xLongrightarrow{\hspace{5mm}}$}; 
\node[right] at (thm.east) {$\xLongrightarrow{\hspace{11mm}}$};
\node[left] at (thm.west) {$\xLongleftarrow{\hspace{11mm}}$};
\node[rotate=90] at ($(thm.south)!0.5!(Liu-Ning.north)$) {$\xLongleftarrow{\hspace{4mm}}$};
\node[right] at (Liu-Ning.east) {$\xLongrightarrow{\hspace{11mm}}$};
\node[left] at (Liu-Ning.west) {$\xLongleftarrow{\!\!\text{Theorem \ref{thm:lowerbound-c(e)}}}$};
\node[rotate=90] at ($(vertex-degree-localized.south)!0.5!(localized-Wilf.north)$){$\xLongleftarrow{\hspace{4mm}}$};
\node[right] at ($(vertex-degree-localized.south)!0.5!(localized-Wilf.north)$) {Theorem \ref{thm:local-Turan-degree}};
\node[rotate=35] at ($(localized-Wilf.south west)!0.5!(Adak-Chandran.north east)$){$\xLongleftarrow{\hspace{13mm}}$};
\node[right] at (Niki.east) {$\xLongrightarrow{\hspace{11mm}}$};
\node[rotate=35] at ($(vertex-degree-localized.south west)!0.5!(Stanley.north east)$) {$\xLongleftarrow{\hspace{11mm}}$};
\node[rotate=35] at ($(Liu-Ning.south west)!0.5!(Niki.north east)$) {$\xLongleftarrow{\hspace{13mm}}$};
\node[rotate=145] at ($(Hong.south east)!0.5!(Stanley.north west)$) {$\xLongleftarrow{\hspace{13mm}}$};
\end{tikzpicture}
\end{center}

Setting $w(e)=1$ in our main theorem, we immediately obtain Corollary \ref{coro:Liu-Ning-26}. 
Moreover, Corollary \ref{coro:Liu-Ning-26} implies Corollary \ref{coro:Niki-spectral-Turan-edge}, 
since $\cl(e)\leq r$ for every edge $e$, also given in \cite{LiuN26}. 

\begin{proof}[Proof of Theorem \ref{thm:local-Turan-degree}]
Recall that $W(G)$ is the $n\times n$ matrix whose $(i,j)$-entry is 
$\big(\frac{\cl(i)}{\cl(i)-1} + \frac{\cl(j)}{\cl(j)-1}\big)/2$. Observe that
\begin{align*}
\sum_{v\in V(G)} d(v)\cdot \frac{\cl(v) - 1}{\cl(v)}
& = \sum_{uv\in E(G)} \bigg(\frac{\cl(u) - 1}{\cl(u)} + \frac{\cl(v) - 1}{\cl(v)}\bigg) \\
& = \sum_{uv\in E(G)} \bigg(\frac{\cl(u)}{\cl(u) - 1} + \frac{\cl(v)}{\cl(v) - 1}\bigg) \cdot
\frac{(\cl(u) - 1) (\cl(v) - 1)}{\cl(u) \cl(v)} \\
& = \bm{x}^{\mathrm{T}} \big(W(G)\circ A(G)\big) \bm{x}, 
\end{align*}
where the vector $\bm{x}\in\mathbb{R}^n$ is defined by
\[
\bm{x} = \bigg(\frac{\cl(v_1)-1}{\cl(v_1)}, \frac{\cl(v_2)-1}{\cl(v_2)}, \ldots, \frac{\cl(v_n)-1}{\cl(v_n)}\bigg)^{\mathrm{T}}.
\]
It follows from Lemma \ref{lem:generalized-MS-inequality} that
\begin{align}
\sum_{v\in V(G)} d(v)\cdot \frac{\cl(v) - 1}{\cl(v)} 
& = \bm{x}^{\mathrm{T}} \big(W(G)\circ A(G)\big) \bm{x} \nonumber \\ 
& = \Big(\frac{\bm{x}}{\|\bm{x}\|_1}\Big)^{\mathrm{T}} \big(W(G)\circ A(G)\big) 
\bigg(\frac{\bm{x}}{\|\bm{x}\|_1}\bigg) \cdot \|\bm{x}\|_1^2 \nonumber \\ 
& \leq \|\bm{x}\|_1^2 = \Bigg(\sum_{v\in V(G)} \frac{\cl(v) - 1}{\cl(v)}\Bigg)^2. \label{eq:auxiliary-eq-5}
\end{align}

Next, we consider the equality in \eqref{eq:auxiliary-eq-5}. From the proof above, equality 
holds if and only if $(\bm{x}/\|\bm{x}\|_1)^{\mathrm{T}} (W(G)\circ A(G)) (\bm{x}/\|\bm{x}\|_1) = 1$.
If $G$ is, up to isolated vertices, a regular complete multi-partite graph, it is straightforward 
to verify that equality holds. Conversely, if equality holds, then by Lemma \ref{lem:generalized-MS-inequality}, 
$G[\supp(\bm{x})]$ is a complete $\omega$-partite graph whose vertex classes $V_1,V_2,\ldots,V_{\omega}$ 
satisfy $\|\bm{1}_{V_i}\circ\bm{x}\|_1 = 1/\omega$ for all $i\in [\omega]$. Hence, for any $i\in [\omega]$,
\[
\frac{1}{\omega} = \|\bm{1}_{V_i}\circ\bm{x}\|_1 = \frac{1}{\|\bm{x}\|_1}\sum_{v\in V_i} \frac{\cl(v) -1 }{\cl(v)} 
= \Big(1 - \frac{1}{\omega}\Big)\cdot \frac{|V_i|}{\|\bm{x}\|_1}. 
\]
Consequently, $G$ is, up to isolated vertices, a regular complete multi-partite graph.
This completes the proof of Theorem \ref{thm:local-Turan-degree}.
\end{proof}

\begin{proof}[Proof of Corollary \ref{thm:AC}]
The proof follows from Corollary \ref{coro:spectral-localized-Turan-vertex} and 
the well-known spectral inequality that the spectral radius of a graph bounds its average degree from below.
\end{proof}

It is known that Hong's inequality implies Stanley's inequality. We show, however, 
that our Corollary \ref{coro:spectral-vertex-degree-localized} is also stronger than Stanley's inequality (Corollary \ref{coro:Stanley}).

\begin{proof}[Proof of Corollary \ref{coro:Stanley}]
(Stanley's inequality) Assuming Corollary \ref{coro:spectral-vertex-degree-localized} 
(which can be proved using only Theorem \ref{thm:main}; see the details later), we have
\[
\lambda (G)^2 \leq 2m-\sum_{v\in V(G)}\frac{d(v)}{\cl(v)}
\leq 2m-\sum_{v\in V(G)}\frac{d(v)}{d(v) + 1}.
\]
The last inequality holds as $\cl(v)$ by $d(v)+1$ for each $v\in V(G)$.
Hence, to complete the proof, it suffices to prove that 
\[
\sqrt{2m-\sum_{v\in V}\frac{d(v)}{d(v) + 1}}\leq -\frac{1}{2}+\sqrt{2m+\frac{1}{4}},
\] 
which is equivalent to Lemma \ref{Lemma:Main}.
\end{proof}

\begin{proof}[Proof of Corollary \ref{coro:Hong}]
(Hong's inequality) Let $G_1,G_2,\ldots,G_t$ be the connected components of $G$, where $t\geq 1$
and $|V(G_i)|\geq 2$ for all $i\in [t]$. Assume without loss of generality that
$\lambda(G_1) = \lambda (G)$. By Corollary \ref{coro:Liu-Ning-26}, we have
\[
\lambda (G)^2 = \lambda (G_1)^2 \leq 2\sum_{e\in E(G_1)} \frac{\cl_{G_1}(e) - 1}{\cl_{G_1}(e)}
 = 2e(G_1) - 2\sum_{e\in E(G_1)} \frac{1}{\cl_{G_1}(e)}.
\]
Note that $2e(G_i) - |V(G_i)| \geq 0$ for all $i\in [t]$. Hence,
\begin{align*}
\lambda(G)^2 
& \leq 2e(G_1) - 2\sum_{e\in E(G_1)} \frac{1}{\cl_{G_1}(e)} + \sum_{i=2}^t \big(2e(G_i) - |V(G_i)|\big) \\
& = 2m - \Bigg(2\sum_{e\in E(G_1)} \frac{1}{\cl_{G_1}(e)} + \sum_{i=2}^t |V(G_i)|\Bigg) \\
& \leq 2m - n + 1,
\end{align*}
where the last inequality follows by applying \eqref{eq:sum_ce_lower_bound} from 
Theorem \ref{thm:lowerbound-c(e)} to the graph $G_1$.
\end{proof}

\begin{remark}
In \cite{Nikiforov20}, Nikiforov demonstrated that Corollary \ref{coro:Niki-spectral-Turan-edge}, 
combined with Tur\'an's Theorem, implies Stanley's inequality. However, Corollary \ref{coro:Niki-spectral-Turan-edge} 
cannot imply Hong's inequality, as numerous counterexamples exist. It is precisely 
this gap that our approach bridges, underscoring its fundamental difference.
\end{remark}

\begin{proof}[Proof of Corollary \ref{coro:spectral-vertex-degree-localized}]
By Corollary \ref{coro:Liu-Ning-26}, we have
\begin{equation}\label{eq:auxiliary-inequality-1}
\lambda (G)^2 \leq 2 \sum_{e\in E(G)} \frac{\cl(e) - 1}{\cl(e)}.
\end{equation}
Observe that for any edge $e=uv\in E(G)$, $\cl(e)\leq \cl(u)$ and $\cl(e)\leq \cl(v)$. Thus, 
\begin{equation}\label{eq:auxiliary-inequality-2}
\lambda (G)^2\leq \sum_{e\in E(G)}\left(\frac{\cl(v)-1}{\cl(v)}+\frac{\cl(u)-1}{\cl(u)}\right) 
= \sum_{v\in V(G)} d(v)\cdot \frac{\cl(v)-1}{\cl(v)}.
\end{equation}

If equality holds in \eqref{eq:local-bound-vertex}, then equalities must hold in 
both \eqref{eq:auxiliary-inequality-1} and \eqref{eq:auxiliary-inequality-2}. 
Consequently, we have (i) $G$ is a complete bipartite graph for $\omega = 2$, 
or a complete regular $\omega$-partite graph for $\omega \geq 3$, and 
(ii) $\cl(e) = \cl(v)$ for any $e\in E(G)$ and $v\in V(G)$. 
It is straightforward to verify that any graph satisfying (i) also satisfies (ii). 
The converse implication is obvious. This completes the proof.
\end{proof}

The result above provides a spectral, vertex-localized version of Tur\'an's theorem.

\begin{proof}[Proof of Corollary \ref{coro:spectral-localized-Turan-vertex}]
By Corollary \ref{coro:spectral-vertex-degree-localized} we obtain 
\[
\lambda (G)^2 \leq \sum_{v\in V(G)} d(v)\cdot\frac{\cl(v)-1}{\cl(v)}.
\]
Combining this with Theorem \ref{thm:local-Turan-degree} gives the desired result.
\end{proof}

Since $\chi(G)\geq \omega(G)$ holds for any graph, Corollary \ref{coro:weighted-Edwards-Elphick} 
follows directly from Theorem \ref{thm:main}.

\begin{proof}[Proof of Corollary \ref{coro:weighted-Cvetkovic}]
Applying Corollary \ref{coro:weighted-Edwards-Elphick} yields 
\[
\chi(G) \geq 1 + \frac{\lambda(G)^2}{\|A(G)\|_F^2-\lambda(G)^2}
\geq 1 + \frac{\lambda(G)^2}{2\sum_{e\in E(G)} w(e) - \lambda(G)^2}.
\]
By Rayleigh’s inequality, $2\sum_{e\in E(G)} w(e) \leq n\cdot\lambda (G)$. Substituting this, we find
\[
\chi (G) \geq 1 + \frac{\lambda(G)^2}{n\cdot\lambda (G) - \lambda(G)^2} 
= 1 + \frac{\lambda(G)}{n - \lambda(G)},
\]
as desired.
\end{proof}

\section{Two technical results}
\label{sec:proofs-technical-results}

In this section, we present two technical results and the proofs. The first one is a main tool 
for proving Corollary \ref{coro:Hong}. For the second one, it supports the proofs of 
Corollaries \ref{coro:Stanley} and \ref{coro:spectral-localized-Turan-vertex}.

For a graph $G$, we say that $G$ satisfies \emph{Property P} if for every maximal clique 
$S$ of $G$ and every vertex $v\in S$, the vertex $v$ has at least one neighbor outside of $S$.

\begin{lemma}\label{Lem:Property-P}
Let $G$ be a graph on $n$ vertices without isolated vertices. If $G$ has Property P, then 
\[
\sum_{e\in E(G)} \frac{1}{c(e)} \geq\frac{n}{2}.
\]
\end{lemma}

\begin{proof}
Since $G$ has Property P, for each edge $e=uv\in E(G)$, we have $c(e)\leq\min\{d(u), d(v)\}$.
This implies that $2/\cl(e) \geq 1/d(u) + 1/d(v)$. Summing this inequality over all edges yields
\[
\sum_{e\in E(G)} \frac{2}{\cl(e)} \geq \sum_{e=uv\in E(G)} \Big(\frac{1}{d(u)}+\frac{1}{d(v)}\Big)
= \sum_{v\in V(G)} \frac{d(v)}{d(v)} = n,
\]
which completes the proof of the lemma.
\end{proof}

Next, we prove Theorem \ref{thm:lowerbound-c(e)}, which is the key tool for proving that our main 
result implies Hong's theorem. This theorem may also be of independent interest.

\begin{proof}[Proof of Theorem \ref{thm:lowerbound-c(e)}]
We prove only inequality \eqref{eq:sum_ce_lower_bound}, as the proof for \eqref{eq:sum_ce_minor_lower_bound} 
is similar. Furthermore, it suffices to consider the case where $G$ is connected.

We proceed by induction on $n$. The base case $n=2$ is trivial. Now, assume $n\geq 3$ and
that the result holds for all graphs of order less than $n$. If $G$ has Property P, 
then the conclusion follows immediately from Lemma \ref{Lem:Property-P}. 
We may therefore assume that $G$ has no Property P. Consequently, there exist 
a maximal clique $S$ and a vertex $w\in S$ such that $N_G(w)\subseteq S$. 



We first assume that $G - w$ is disconnected. Let $H_1,H_2,\ldots,H_t$ be all components of $G-w$. 
Let $G_i = G[V(H_i)\cup \{w\}]$. Obviously, $E(G_i)\neq \emptyset$. By induction, we have 
\[
\sum_{e\in E(G_i)} \frac{1}{\cl_{G_i}(e)} \geq \frac{|H_i|}{2}
\]
for each $i\in [t]$. Summing up all these inequalities, with help of the obvious fact 
$\cl_{G_i}(e) = \cl_{G}(e)$ for any $e\in E(G_i)$, we have
\[
\sum_{e\in E(G)} \frac{1}{\cl_G(e)}\geq \sum_{i=1}^t \frac{|H_i|}{2} = \frac{n-1}{2}.
\]

Now, assume $G - w$ is connected. Observe that 
\[
\sum_{e\in E(G)} \frac{1}{\cl_G(e)} - \sum_{e\in E(G-w)} \frac{1}{\cl_{G-w}(e)} 
= \sum_{e:\, w\in e} \frac{1}{\cl_G(e)} + \sum_{e\in E(S-w)} \Big(\frac{1}{\cl_G(e)} - \frac{1}{\cl_{G-w}(e)}\Big).
\]
Since for each $e\in E(S-w)$ we have $\cl_{G-w}(e) \geq \cl_G(e) - 1$, the inductive hypothesis yields
\begin{align*}
\sum_{e\in E(G)} \frac{1}{\cl_G(e)} 
& = \sum_{e\in E(G-w)} \frac{1}{\cl_{G-w}(e)} + \sum_{e:\, w\in e} \frac{1}{\cl_G(e)} + \sum_{e\in E(S-w)} \Big(\frac{1}{\cl_G(e)} - \frac{1}{\cl_{G-w}(e)}\Big) \\
& \geq \frac{n-2}{2} + \frac{|S|-1}{|S|} + \sum_{e\in E(S-w)} \Big(\frac{1}{\cl_G(e)} - \frac{1}{\cl_G(e) - 1}\Big) \\
& = \frac{n-2}{2} + \frac{|S|-1}{|S|} - \binom{|S|-1}{2} \cdot \frac{1}{|S|(|S|-1)} \\
& = \frac{n-1}{2},
\end{align*}
completing the proof.
\end{proof}


\begin{lemma}\label{Lemma:Main}
Let $G$ be a graph on $n$ vertices and $m$ edges. Then
\begin{equation}\label{eq:auxiliary-eq}
\Bigg(n - \sum_{v\in V(G)} \frac{1}{d_G(v) + 1}\Bigg) 
\Bigg(n + 1 - \sum_{v\in V(G)} \frac{1}{d_G(v)+1}\Bigg) \geq 2m.
\end{equation}
\end{lemma}

\begin{proof}
We proceed by induction on $n$. The validity of \eqref{eq:auxiliary-eq} for all graphs with
$n\leq 5$ can be verified by a computer program, establishing the base case. 

Now assume that the lemma holds for all graphs of order $n-1$, where $n\geq 6$. 
Let $u\in V(G)$ with $d_G(u) = \delta (G) =:\delta$, and
define $G' = G - u$. In the following, if there is no danger of ambiguity, we use $d(v)$ 
and $d'(v)$ instead of $d_G(v)$ and $d_{G'}(v)$, respectively. For brevity, set
\[
\psi_G(n) := n^2 + n - (2n+1)\sum_{v\in V(G)} \frac{1}{d(v)+1} + \Bigg(\sum_{v\in V(G)} \frac{1}{d(v)+1}\Bigg)^2.
\]
By the induction hypothesis applied to $G'$, we have $\psi_{G'}(n-1) \geq 2(m-\delta)$. 
We now analyze the difference $\psi_G(n) - \psi_{G'}(n-1)$. Starting from the definition:
\begin{align*}
& ~ \psi_G(n) - \psi_{G'}(n-1) \\
= & ~ 2n {-} \Bigg(\sum_{v\in V(G)} \frac{2n + 1}{d(v) {+} 1} {-} \sum_{v\in V(G')} \frac{2n - 1}{d'(v) {+} 1}\Bigg) 
+ \Bigg(\sum_{v\in V(G)} \frac{1}{d(v) {+} 1}\Bigg)^2 {-} \Bigg(\sum_{v\in V(G')} \frac{1}{d'(v) {+} 1}\Bigg)^2.
\end{align*}
A careful calculation yields  
\begin{align*}
& ~ \psi_G(n) - \psi_{G'}(n-1) \\
= & ~ 2n {-} \Bigg[\Bigg(\frac{2n - 1}{\delta + 1} {+} \sum_{v\in V(G')} \frac{2n - 1}{d(v) {+} 1}\Bigg)
{+} \sum_{v\in V(G)} \frac{2}{d(v) {+} 1} {-} \sum_{v\in V(G')} \frac{2n - 1}{d'(v) {+} 1}\Bigg] \\
& ~ + \Bigg(\frac{1}{\delta + 1} + \sum_{v\in V(G')} \frac{1}{d(v) + 1}\Bigg)^2 - \Bigg(\sum_{v\in V(G')} \frac{1}{d'(v) + 1}\Bigg)^2 \\
= & ~ 2n + (2n - 1) \Bigg(\sum_{v\in V(G')} \frac{1}{d'(v) + 1} - \sum_{v\in V(G')} \frac{1}{d(v) + 1}\Bigg) - \frac{2n-1}{\delta + 1} - \sum_{v\in V(G)} \frac{2}{d(v) + 1} \\ 
& ~ + \Bigg(\sum_{v\in V(G')} \frac{1}{d(v)+1}\Bigg)^2 - \Bigg(\sum_{v\in V(G')} \frac{1}{d'(v)+1}\Bigg)^2 + \frac{2}{\delta + 1} \sum_{v\in V(G')} \frac{1}{d(v) + 1} + \frac{1}{(\delta + 1)^2}.
\end{align*}
We now derive a lower bound on $\psi_G(n) - \psi_{G'}(n-1)$. After ignoring the item 
$\frac{2}{\delta + 1} \sum_{v\in V(G')} \frac{1}{d(v) + 1}$, we immediately obtain
\begin{align*}
& ~ \psi_G(n) - \psi_{G'}(n-1) \\
\geq & ~ \Bigg(\sum_{v\in V(G')} \frac{1}{d'(v) {+} 1} {-} \sum_{v\in V(G')} \frac{1}{d(v) {+} 1}\Bigg) 
\Bigg(2n-1 - \sum_{v\in V(G')} \frac{1}{d(v) {+} 1} {-} \sum_{v\in V(G')} \frac{1}{d'(v) + 1}\Bigg) \\
& ~ + 2n - \frac{2n-1}{\delta + 1} + \frac{1}{(\delta + 1)^2}- 2\sum_{v\in V(G)} \frac{1}{d(v) {+} 1}. 
\end{align*}
It follows from $\psi_{G'}(n-1) \geq 2(m-\delta)$ that 
\begin{align*}
\psi_G(n) & \geq 2m {+} \Bigg(\sum_{v\in V(G')} \bigg(\frac{1}{d'(v) {+} 1} {-} \frac{1}{d(v) {+} 1}\bigg)\Bigg) 
\Bigg(2n {-} 1 {-} \sum_{v\in V(G')} \bigg(\frac{1}{d(v) {+} 1} {+} \frac{1}{d'(v) {+} 1}\bigg)\Bigg) \\
& ~~~ + 2n -2 \delta - \frac{2n-1}{\delta + 1} + \frac{1}{(\delta + 1)^2} - 2\sum_{v\in V(G)} \frac{1}{d(v) + 1}.
\end{align*}
Note that $d'(v)\leq d(v)$ for each $v\in V(G')$. This implies that 
\[
\psi_G(n) \geq 2m + 2n -2 \delta - \frac{2n-1}{\delta + 1} 
+ \frac{1}{(\delta + 1)^2} - 2\sum_{v\in V(G)} \frac{1}{d(v) + 1}.
\]
Since $d(v)\geq\delta$ for all $v\in V(G)$, we obtain
\[
\psi_G(n) \geq 2n-2\delta - \frac{4n - 1}{\delta + 1} + \frac{1}{(\delta + 1)^2}.
\]

To finish the induction step, it suffices to prove that
\[
2n - 2\delta - \frac{4n - 1}{\delta + 1} + \frac{1}{(\delta + 1)^2} \geq 0.
\]
Toward this goal, we define the functions
\[
f(x) := 2n - 2x - \frac{4n-1}{x+1} \quad \text{and} \quad g(x) := f(x) + \frac{1}{(x + 1)^2},
\]
so that the target inequality is equivalent to $g(\delta) \geq 0$.
For $\delta\in [2, n-3]$, the convexity of $f(x)$ implies that 
\[
g(\delta) = f(\delta) + \frac{1}{(\delta+1)^2} \geq \min\{f(2), f(n-3)\} + \frac{1}{(\delta+1)^2}.
\]
A direct computation shows that
\begin{align*}
f(2) & = 2n - 4 - \frac{4n-1}{3} = \frac{2n-11}{3} > 0, \\
f(n-3) & = 2n - 2(n-3) - \frac{4n-1}{n-2} = \frac{2n-8}{n-2} > 0.
\end{align*}
Hence, the inequality holds for $2\leq\delta\leq n-3$.
It remains to consider the cases $\delta\in\{n-2, n-1\}$, 
$\delta = 1$, and $\delta = 0$.

{\bfseries Case 1.} $\delta\in\{n-2, n-1\}$. If $\delta = n-1$, 
then $G=K_n$, and it is straightforward to verify that \eqref{eq:auxiliary-eq} holds. 
If $\delta = n-2$, then 
\begin{align*}
\psi_G(n) 
& = \Bigg(n - \sum_{v\in V(G)} \frac{1}{d(v) + 1}\Bigg) 
\Bigg(n + 1 - \sum_{v\in V(G)} \frac{1}{d(v) + 1}\Bigg) \\
& \geq \bigg(n - \frac{n}{n - 1}\bigg) 
\bigg(n + 1 - \frac{n}{n - 1}\bigg) \\
& = n^2 - n - 2 - \frac{1}{n-1} + \frac{1}{(n-1)^2} \\
& > n^2 - n - \frac{11}{5}.
\end{align*}
Hence, if $2m \leq n^2 - n - 3$, then \eqref{eq:auxiliary-eq} holds. 
It remains to consider the situation $2m = n^2 - n - 2$. 
When $2m = n^2 - n - 2$, the degree sequence of $G$ is 
\[
n-1,\ldots,n-1,n-2,n-2,n-2,n-2.
\]
A direct calculation then yields
\begin{align*}
\psi_G(n) & = n^2 + n - (2n+1) \bigg(\frac{n-4}{n} + \frac{4}{n-1}\bigg)
+ \bigg(\frac{n-4}{n} + \frac{4}{n-1}\bigg)^2 \\
& \geq n^2 - n - 2 = 2m,
\end{align*}
so \eqref{eq:auxiliary-eq} is satisfied in this case as well. 

{\bfseries Case 2.} $\delta=1$. Let $N_G(u)=\{w\}$. Recall that 
\[
\psi_{G-u}(n-1) = (n-1)n - (2n-1)\Bigg(\sum_{v\in V(G)\setminus\{u,w\}} \frac{1}{d(v)+1} + \frac{1}{d(w)}\Bigg) 
+ \Bigg(\sum_{v\in V(G)\setminus\{u,w\}} \frac{1}{d(v)+1} + \frac{1}{d(w)}\Bigg)^2.
\]
A straightforward computation yields
\begin{align*}
& ~ \psi_G(n) - \psi_{G-u}(n-1) \\
= & ~ 2n - (2n-1) \Bigg(\frac{1}{2} - \frac{1}{d(w)(d(w)+1)}\Bigg) - 2\sum_{v\in V(G)} \frac{1}{d(v) + 1} \\
& ~ + \Bigg(\frac{1}{2} - \frac{1}{d(w)(d(w)+1)}\Bigg) \Bigg(\sum_{v\in V(G)} \frac{1}{d_G(v)+1} 
+ \sum_{v\in V(G)\setminus\{u,w\}} \frac{1}{d(v)+1} + \frac{1}{d(w)}\Bigg) \\
& = 2n - \Bigg(\frac{1}{2} - \frac{1}{d(w)(d(w)+1)}\Bigg) \Bigg(2n - \frac{3}{2} 
- 2\sum_{v\in V(G)\setminus\{u\}} \frac{1}{d(v)+1}\Bigg) - 2\sum_{v\in V(G)} \frac{1}{d(v) + 1} \\
& \geq 2n - \Bigg(n - \frac{3}{4} - \sum_{v\in V(G)\setminus\{u\}} \frac{1}{d(v)+1}\Bigg) - 2\sum_{v\in V(G)} \frac{1}{d(v) + 1} \\
& = n - \frac{1}{4} - \sum_{v\in V(G)\setminus\{u\}} \frac{1}{d(v)+1}.
\end{align*}
It follows that 
\[
\psi_G(n) \geq \psi_{G-u}(n-1) + n - \frac{1}{4} - \sum_{v\in V(G)\setminus\{u\}} \frac{1}{d(v)+1}.
\]
By the induction hypothesis, $\psi_{G-u}(n-1) \geq 2(m-1)$. Furthermore, since
each term in $\sum_{v\in V(G)\setminus\{u\}} \frac{1}{d(v)+1}$ is at most $1/2$. 
Applying these bounds gives
\[
\psi_G(n) \geq 2(m - 1) + n - \frac{1}{4} - \frac{n-1}{2} = 2m + \frac{2n - 7}{4} > 2m.
\]

{\bfseries Case 3.} $\delta = 0$. Let $N_G(u) = \emptyset$. We have
$\psi_G(n) - \psi_{G\setminus\{u\}} (n-1) = 0$. Hence, $\psi_G(n) = \psi_{G\setminus\{u\}} (n-1) \geq 2m$.

This completes the proof of Lemma \ref{Lemma:Main}.
\end{proof}

\section*{Acknowledgment}
The second author thanks R. Adak and L.S. Chandran for drawing his attention to \cite{AC-25arxiv}. 
After the first version of the manuscript was submitted to arXiv, the authors received an email from 
R. Tripathi. The authors thank R. Tripathi for pointing out that Corollary \ref{coro:spectral-localized-Turan-vertex} 
confirms a conjecture of his, which is proposed independently of us, and other potential 
applications of our results in Probability Theory and Operator Algebras.

\end{document}